\RequirePackage{etex}
\RequirePackage{easymat}

\documentclass[12pt]{elsarticle}
\usepackage{amsthm,amsmath,amssymb,
extsizes,mathrsfs}
\usepackage[all]{xy}

\usepackage[active]{srcltx}
\DeclareMathOperator{\diagg}{diag}
\renewcommand{\le}{\leqslant}
\renewcommand{\ge}{\geqslant}

\newtheorem{theorem}{Theorem}
\newtheorem{lemma}[theorem]{Lemma}

\theoremstyle{remark}

\begin{document}

\title{Remarks on the classification
of a pair of commuting semilinear operators}

\author[deb]{Debora Duarte de Oliveira}
\ead{d.duarte.oliveira@gmail.com}

\author[horn]{Roger A. Horn}
\ead{rhorn@math.utah.edu}

\author[klim]{Tatiana Klimchuk}
\ead{klimchuk.tanya@gmail.com}

\author[serg]{Vladimir V. Sergeichuk\corref{cor}}
\ead{sergeich@imath.kiev.ua}

\address[deb]{Department of Mathematics, University
of S\~ao Paulo, Brazil.}

\address[horn]
{Department of Mathematics, University
of Utah, Salt Lake City, Utah, USA.}

\address[klim]
{Faculty of Mechanics and Mathematics,
Kiev National Taras Shevchenko
University, Kiev, Ukraine.}

\address[serg]
{Institute of Mathematics,
Tereshchenkivska 3, Kiev, Ukraine.}

\cortext[cor]{Corresponding author.}

\begin{abstract}
Gelfand and Ponomarev [Functional Anal.
Appl. 3 (1969) 325--326] proved that
the problem of classifying pairs of
commuting linear operators contains the
problem of classifying $k$-tuples of
linear operators for any $k$. We prove
an analogous statement for semilinear
operators.
\end{abstract}


\maketitle

\section{Introduction}

Gelfand and Ponomarev \cite{gel-pon}
proved that the problem of classifying
pairs of commuting linear operators on
a vector space contains the problem of
classifying $k$-tuples of linear
operators for any $k$ (that is, the
solution of the former problem would
imply the solution of the latter
problem).

We prove an analogous statement for
semilinear operators. A mapping ${\cal
A}: U\to V$ between two complex vector
spaces is called \emph{semilinear} if
\begin{equation*}\label{lke}
{\cal A}(u+u')={\cal A}u+{\cal A}u',\qquad
{\cal A}(\alpha u)=\bar\alpha {\cal A}u
\end{equation*}
for all $u,u'\in U$ and $\alpha
\in\mathbb C$. We write ${\cal A}:
U\dashrightarrow V$ if $\cal A$ is
semilinear. If $U=V$ then $\cal A$ is
called a \emph{semilinear operator}. In
Section \ref{s1} we recall some basic
facts about semilinear mappings and
describe all semilinear operators that
commute with a given nilpotent
semilinear operator.

In Sections \ref{s2} and \ref{s3} we
prove the following theorem, which
extends the results of \cite{gel-pon}
to semilinear operators.

\begin{theorem}\label{gey}
{\rm(a)} The problem of classifying
pairs of commuting semilinear operators
contains the problem of classifying
pairs of arbitrary semilinear
operators.

{\rm(b)} The problem of classifying
pairs of semilinear operators contains
the problem of classifying
$(p+q)$-tuples consisting of $p$ linear
operators and $q$ semilinear operators,
in which $p$ and $q$ are arbitrary
nonnegative integers.
\end{theorem}

A similar statement for operators on
unitary spaces was proved in
\cite[Lemma 2]{g-h-s}: the problem of
classifying semilinear operators on a
unitary space contains the problem of
classifying tuples of linear and
semilinear operators on a unitary
space.

Any tuple in Theorem \ref{gey}(b)
consists of operators acting on the
same vector space. In Section \ref{s4}
we generalize Theorem \ref{gey}(b) to
collections of mappings that act on
different spaces. We use the notion of
biquiver representations introduced in
\cite[Section 5]{ser_surv}, which
generalizes the notion of quiver
representations introduced by Gabriel
\cite{gab}. A \emph{biquiver} is a
directed graph with full and dashed
arrows; for example,
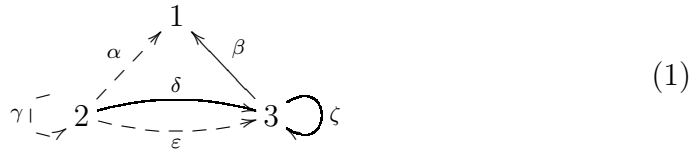
\begin{equation}\label{ksy}
\begin{split}
\xymatrix@=2pc{
 &{1}&\\
 {2}\ar@(ul,dl)@{-->}_{\gamma }
 \ar@{-->}[ur]^{\alpha}
  \ar@/^/@{->}[rr]^{\delta }
 \ar@/_/@{-->}[rr]_{\varepsilon} &&{3}
 \ar[ul]_{\beta}
 \ar@(ur,dr)^{\zeta}
 }\end{split}
\end{equation}
Its \emph{representation} is given by
assigning to each vertex a complex
vector space, to each full arrow a
linear mapping, and to each dashed
arrow a semilinear mapping of the
corresponding vector spaces. Thus, a
representation
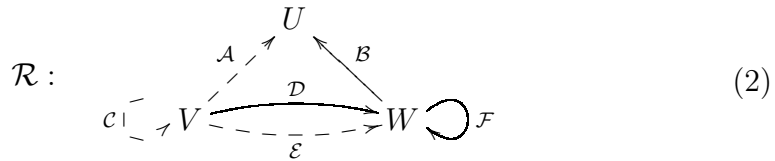
\begin{equation}\label{2.6aa}
\begin{split}
{\cal R}:\quad\raisebox{25pt}{\xymatrix@=2pc{
 &{U}&\\
  \save
!<-2mm,0cm>
 {V}\ar@(ul,dl)@{-->}_{\cal C}
 \restore
 \ar@{-->}[ur]^{\cal A}
  \ar@/^/@{->}[rr]^{\cal D}
 \ar@/_/@{-->}[rr]_{\cal E} &&{W}
 \ar[ul]_{\cal B}
  \save
!<2mm,0cm>
 \ar@(ur,dr)^{\cal F}
\restore}
 }\end{split}
\end{equation}
of \eqref{ksy} is formed by complex
spaces $U,V,W$, linear mappings
\[{\cal B}:W\to U,\quad {\cal D}:V\to W,
\quad {\cal F}:W\to
W,\] and semilinear mappings
\[
{\cal A}:V\dashrightarrow U,\quad
{\cal C}: V\dashrightarrow V,
\quad {\cal E}: V \dashrightarrow W.
\]
A biquiver without dashed arrows is a
quiver and its representations are the
quiver representations.

In Section \ref{s4} we prove the
following generalization of Theorem
\ref{gey}(b).

\begin{theorem}\label{gey2}
The problem of classifying pairs of
semilinear operators contains the
problem of classifying representations
of any biquiver.
\end{theorem}

The results in \cite{gel-pon} ensure
that the problem of classifying pairs
of linear operators over any field
$\mathbb F$ contains the problem of
classifying $k$-tuples of linear
operators. This implies that it
contains the problem of classifying
representations of an arbitrary
$k$-dimensional algebra $\Lambda$ over
$\mathbb F$ by operators of a vector
space\footnote{The exists an
isomorphism from $\Lambda$ to a factor
algebra ${\mathbb F}\langle
x_1,\dots,x_t\rangle/J$ of the free
algebra of noncommutative polynomials
in $x_1,\dots,x_t$. Let $g_1,\dots,g_r$
be generators of $J$, then each
representation of $\Lambda$ is a
$k$-tuple of linear operators $({\cal
A}_1,\dots,{\cal A}_k)$ satisfying
$g_i({\cal A}_1,\dots,{\cal A}_k)=0$
for all $i=1,\dots,r$.}. Thus, the
problem of classifying pairs of linear
operators contains the problem of
classifying representations of any
quiver\footnote{The representations of
a quiver can be identified with the
representations of its path algebra.};
a direct proof of this inclusion is
given in \cite[Sect. 3.1]{ser_can} and
\cite{bel-ser_comp}. The problem of
classifying pairs of linear operators
also contains the problem of
classifying any system of linear
mappings and bilinear or sesquilinear
forms because the latter problem can be
reduced to the problem of classifying
quiver representations (see
\cite{ser_izv,ser_surv,ser_isom}).

For this reason, the problem of
classifying pairs of linear operators
is used in representation theory as a
measure of complexity: all
classification problems split into two
types: {\it tame} (or classifiable) and
{\it wild} (containing the problem of
classifying pairs of linear operators);
wild problems are considered as
hopeless. These terms were introduced
by Donovan and Freislich \cite{don} in
analogy with the partition of animals
into tame and wild ones. It follows
from Theorem \ref{gey2} that the
problem of classifying pairs of
semilinear operators plays the same
role in the theory of systems of linear
and semilinear mappings.

\section{Semilinear operators commuting
with a nilpotent semilinear operator}
\label{s1}

In this section, we describe all
semilinear operators that commute with
a given nilpotent semilinear operator,
but first we recall basic facts about
semilinear mappings. All vector spaces
and matrices that we consider are over
the field of complex numbers.

We denote by $\bar a$ the complex
conjugate of $a\in \mathbb C$, by
$[v]_e$ the coordinate vector of $v$ in
a basis $e_1,\dots,e_n$, and by
$S_{e\to e'}$ the transition matrix
from a basis $e_1,\dots,e_n$ to a basis
$e'_1,\dots,e'_n$.  If $A=[a_{ij}]$
then $\bar A:=[\bar a_{ij}]$.

Let ${\cal A}: U\dashrightarrow V$  be
a semilinear mapping. We say that an
$m\times n$ matrix ${\cal A}_{fe}$ is
the \emph{matrix of $\cal A$ in bases}
$e_1,\dots,e_n$ of $U$ and
$f_1,\dots,f_m$ of $V$ if
\begin{equation}\label{feo}
[{\cal A}u]_f=\overline{{\cal A}_{fe}[u]_e}\qquad
\text{for all }u\in U.
\end{equation}
Therefore, the columns of ${\cal
A}_{fe}$ are $\overline{[{\cal
A}e_1]_f}, \dots, \overline{[{\cal
A}e_n]_f}$. We write ${\cal A}_{e}$
instead of ${\cal A}_{ee}$ if $U=V$.

If $e'_1,\dots,e'_n$ and
$f'_1,\dots,f'_m$ are other bases of
$U$ and $V$, then
\begin{equation*}\label{swk}
{\cal A}_{f'e'}=\bar S_{f\to f'}^{-1}{\cal A}_{fe}
S_{e\to e'}
\end{equation*}
since the right hand matrix satisfies
\eqref{feo} with $e',f'$ instead of
$e,f$:
\[
\overline{\bar S_{f\to f'}^{-1}{\cal A}_{fe}
S_{e\to e'}[v]_{e'}}=
S_{f\to f'}^{-1}\overline{{\cal A}_{fe}
[v]_{e}}=S_{f\to f'}^{-1}[{\cal A}
v]_f=[{\cal A}
v]_{f'}
\]

In particular, if $U=V$, then
\begin{equation*}\label{swk1}
{\cal A}_{e'}=\bar S_{e\to e'}^{-1}{\cal A}_{e}
S_{e\to e'}
\end{equation*}
and so ${\cal A}_{e'}$ and ${\cal
A}_{e}$ are consimilar: recall that two
matrices $A$ and $B$ are
\emph{consimilar} if there exists a
nonsingular matrix $S$ such that $\bar
S^{-1}AS=B$ (see \cite[Section
4.6]{hor_John}). Two pairs $(A_1,A_2)$
and $(B_1,B_2)$ of $n\times n$ matrices
are called \emph{consimilar} if there
exists a nonsingular matrix $S$ such
that \[\bar S^{-1}(A_1, A_2)S:=(\bar
S^{-1}A_1S, \bar
S^{-1}A_2S)=(B_1,B_2).\]

Thus, the problem of classifying pairs
of semilinear operators reduces to the
problem of classifying matrix pairs up
to consimilarity.

\begin{lemma}\label{kow}
The composition of two semilinear
operators ${\cal A}: U\dashrightarrow
U$ and ${\cal B}: U\dashrightarrow U$
is a linear operator and its matrix in
a basis $e_1,\dots,e_n$ of $U$ is
\begin{equation}\label{y49}
({\cal AB})_e=\bar{\cal A}_e{\cal B}_e
\end{equation}
\end{lemma}

\begin{proof}
The identity \eqref{y49} follows from
observing that ${\cal AB}$ is a linear
operator and
\[
\bar{\cal A}_e{\cal B}_e[u]_e=
\overline{{\cal A}_e[{\cal B}u]_e}=
[{\cal A}({\cal B}u)]_e=
[({\cal A}{\cal B})u]_e\qquad
\text{for each $u\in U$}.
\]
\end{proof}

A canonical form of a matrix under
consimilarity is given in \cite[Theorem
3.1]{hon-hor}. In particular, each
nilpotent matrix is consimilar to a
nilpotent Jordan matrix that is
determined uniquely up to permutation
of Jordan blocks. Each nilpotent Jordan
matrix is \emph{permutationally
similar} (i.e., is reduced by
simultaneous permutations of rows and
columns) to the form
\begin{equation}\label{jer}
J:=J_{p_1}(0_{q_1})\oplus\dots\oplus
J_{p_t}(0_{q_t}),\qquad
p_i\ne p_j\text{ if }i\ne j,
\end{equation}
in which
\[
J_{p_i}(0_{q_i}):=\begin{bmatrix}
     0_{q_i}&I_{q_i}&&0\\&0_{q_i}&\ddots
     \\&&\ddots&I_{q_i}
     \\0&&&0_{q_i}\\
   \end{bmatrix}\qquad(p_i\times p_i\text{
   subblocks of size }q_i\times q_i).
\]
We consider $J$ as a block matrix
$[J_{ij}]_{i,j=1}^t$; each block
$J_{ij}$ is $ p_iq_i\times p_jq_j$ and
is partitioned into $p_i\times p_j$
subblocks of size $q_i\times q_j$.

All matrices that commute with a given
square matrix are described in
\cite[Sect. VIII, \S 2]{gan}. In the
following lemma, we give an analogous
description of all matrices $S$
satisfying $\bar SJ=JS$.

\begin{lemma}\label{kse}
{\rm(a)} For each nilpotent semilinear
operator ${\cal J}:U\dashrightarrow U$
there exists a basis in which its
matrix has the form \eqref{jer}. If
${\cal S}:U\dashrightarrow U$ is
another semilinear operator and $S$ is
its matrix  in the same basis, then
${\cal S}{\cal J}={\cal J}{\cal S}$ if
and only if $\bar S J= JS$.

{\rm(b)} Let $J$ be the matrix
\eqref{jer}, let $S$ be a matrix of the
same size, and let $S$ be partitioned
into blocks and subblocks conformally
to the partition of $J$. Then $\bar S
J= JS$ if and only if
$S=[S_{ij}]_{i,j=1}^t$, in which every
$S_{ij}$ is a $p_iq_i\times p_jq_j$
block of the form
\begin{equation}\label{gny}
S_{ij}=
  \begin{cases}
   \begin{bmatrix}
      &\quad&C_{ij}&C'_{ij}&C''_{ij}&C'''_{ij}&\dots
      &C_{ij}^{(p_i-1)}\\
      &&& \bar C_{ij} & \bar C'_{ij}&
      \bar C''_{ij}
      &\dots&\bar C_{ij}^{(p_i-2)} \\
      &&&&C_{ij}&C'_{ij}&\dots&C_{ij}^{(p_i-3)}\\
&&&&&\bar C_{ij}&\dots&\bar C_{ij}^{(p_i-4)}\\
      &&&&&&\ddots &  \vdots \\
0&&&& & &&\hat C_{ij} \\
   \end{bmatrix} & \text{if }p_i\le p_j, \\[65pt]
\begin{bmatrix}
      C_{ij}&C'_{ij}&C''_{ij}&C'''_{ij}
      &\dots&C_{ij}^{(p_j-1)}\\
      & \bar C_{ij} & \bar C'_{ij}
      & \bar C''_{ij}&\dots&\bar C_{ij}^{(p_j-2)} \\
      &&C_{ij}&C'_{ij}&\dots&C_{ij}^{(p_j-3)} \\
      &&&\bar C_{ij}&\dots&\bar C_{ij}^{(p_j-4)} \\
      &&&&\ddots &  \vdots \\
&&&  &&\hat C_{ij} \\[5mm]
0
   \end{bmatrix} & \text{if }p_i\ge p_j
  \end{cases}
\end{equation}
and \[\hat C_{ij}=
  \begin{cases}
    C_{ij} & \text{if $\min(p_i,p_j)$ is odd}, \\
    \bar C_{ij} & \text{if $\min(p_i,p_j)$ is even}.
  \end{cases}\]
\end{lemma}

For example, if $J=J_4(0_q)\oplus
J_2(0_{q'})$ and $\bar S J= JS$, then
\begin{equation}\label{dsu}
\ \begin{matrix}
J= \\ \\
\end{matrix}\
\begin{MAT}(e){ccccccc}
\,0_q\,&I_q&&   &&&\scriptstyle \it 1\!,\!1\\
&0_q&I_q&   &&&\scriptstyle \it 2\!,\!1\\
&&0_q&I_q   &&&\scriptstyle \it 3\!,\!1\\
&&&0_q    &&&\scriptstyle \it 4\!,\!1\\
&&&   &0_{q'}&I_{q'}&\scriptstyle \it 1\!,2\\
&&&   &&0_{q'}&\scriptstyle \it 2\!,2\\
\scriptstyle \it 1\!,\!1&\scriptstyle
\it 2\!,\!1&
\scriptstyle \it 3\!,\!1&\scriptstyle
\it 4\!,\!1&
\scriptstyle \it 1\!,2&\scriptstyle \it 2\!,2&
\addpath{(0,1,4)rrrrrruuuuuulllllldddddd}
\addpath{(4,1,1)uuuuuu}
\addpath{(0,3,1)rrrrrr}\\
\end{MAT}\quad
\begin{matrix}
S=\,\\ \\
\end{matrix}\
\begin{MAT}(e){ccccccc}
C&C_1&C_2&C_3   &D&D_1&\scriptstyle \it 1\!,\!1\\
&\bar C&\bar C_1&\bar C_2   &&\bar D
&\scriptstyle \it 2\!,\!1\\
&&C&C_1   &&&\scriptstyle \it 3\!,\!1\\
&&&\bar C    &&&\scriptstyle \it 4\!,\!1\\
&&E&E_1   &F&F_1&\scriptstyle \it 1\!,2\\
&&&\bar E   &&\bar F&\scriptstyle \it 2\!,2\\
\scriptstyle \it 1\!,\!1&\scriptstyle
\it 2\!,\!1&
\scriptstyle \it 3\!,\!1&\scriptstyle
\it 4\!,\!1&
\scriptstyle \it 1\!,2&\scriptstyle
\it 2\!,2&
\addpath{(0,1,4)rrrrrruuuuuulllllldddddd}
\addpath{(4,1,1)uuuuuu}
\addpath{(0,3,1)rrrrrr}\\
\end{MAT}
\end{equation}
(unspecified blocks are zero).

\begin{proof}[Proof of Lemma \ref{kse}]
(a) This statement follows from Lemma
\ref{kow} and the canonical form of a
matrix under consimilarity
\cite{hon-hor}.

(b) We have $\bar S J= JS$ if and only
if
\begin{equation}\label{lut}
\bar S_{ij}
J_{p_j}(0_{q_j})= J_{p_i}(0_{q_i})
S_{ij}\qquad \text{for all }i,j=1,\dots,t.
\end{equation}

Assuming \eqref{lut}, we verify that
each $S_{ij}$ has the form \eqref{gny}
as follows: divide $S_{ij}$ into
$p_i\times p_j$ subblocks of size
$q_i\times q_j$ and compare subblocks
in the identity $\bar S_{ij}
J_{p_j}(0_{q_j})= J_{p_i}(0_{q_i})
S_{ij}$ starting in subblock $(p_i,1)$,
moving along vertical strips from
bottom to up, and finishing in subblock
$(1,p_j)$.

Conversely, if all $S_{ij}$ have the
form \eqref{gny}, then \eqref{lut}
holds.
\end{proof}

Let $M$ be an arbitrary block matrix
partitioned into strips and substrips
such that all diagonal blocks and
subblocks are square. We index the
$\alpha $th substrip of $i$th strip by
the pair $\alpha ,\!i$ (as in
\eqref{dsu}). Denote by $M^{\#}$ the
block matrix obtained from $M$ by
permuting its substrips so that their
index pairs form a lexicographically
ordered sequence. For example, if $J$
and $S$ are the block matrices
\eqref{dsu}, then
\begin{equation}\label{cev}
\begin{matrix}
J^{\#}=\,\\ \\
\end{matrix}\
\begin{MAT}(e){ccccccc}
\,0_q\,&&I_q&   &&&\scriptstyle \it 1\!,\!1\\
&0_{q'}&&I_{q'}   &&&\scriptstyle \it 1\!,2\\
&&0_q&   &I_q&&\scriptstyle \it 2\!,\!1\\
&&&0_{q'}    &&&\scriptstyle \it 2\!,2\\
&&&   &0_{q}&I_{q}&\scriptstyle \it 3\!,\!1\\
&&&   &&0_{q}&\scriptstyle \it 4\!,\!1\\
\scriptstyle \it 1\!,\!1&\scriptstyle
\it 1\!,2&
\scriptstyle \it 2\!,\!1&\scriptstyle
\it 2\!,2&
\scriptstyle \it 3\!,\!1&\scriptstyle
\it 4\!,\!1&
\addpath{(0,1,4)rrrrrruuuuuulllllldddddd}
\addpath{(4,1,1)uuuuuu}
\addpath{(2,1,1)uuuuuu}
\addpath{(5,1,1)uuuuuu}
\addpath{(0,2,1)rrrrrr}
\addpath{(0,5,1)rrrrrr}
\addpath{(0,3,1)rrrrrr}\\
\end{MAT}\quad
\begin{matrix}
S^{\#}=\,\\ \\
\end{matrix}\
\begin{MAT}(e){ccccccc}
C&D&C_1&D_1   &C_2&C_3&\scriptstyle \it 1\!,\!1\\
& F&&F_1   &E&E_1&
\scriptstyle \it 1\!,2\\
&&\bar C&\bar D &\bar C_1&\bar C_2&\scriptstyle
\it 2\!,\!1\\
&&&\bar F    &&\bar E&\scriptstyle \it 2\!,2\\
&&&   &C&C_1&\scriptstyle \it 3\!,\!1\\
&&&   &&\bar C&\scriptstyle \it 4\!,\!1\\
\scriptstyle \it 1\!,\!1&\scriptstyle \it
1\!,2&
\scriptstyle \it 2\!,\!1&\scriptstyle
\it 2\!,2&
\scriptstyle \it 3\!,\!1&\scriptstyle
\it 4\!,\!1&
\addpath{(0,1,4)rrrrrruuuuuulllllldddddd}
\addpath{(4,1,1)uuuuuu}
\addpath{(2,1,1)uuuuuu}
\addpath{(5,1,1)uuuuuu}
\addpath{(0,2,1)rrrrrr}
\addpath{(0,5,1)rrrrrr}
\addpath{(0,3,1)rrrrrr}\\
\end{MAT}
\end{equation}
The block matrix $M^{\#}$ can be
obtained from $M$ as follows: we gather
at the top the first substrips of all
horizontal strips, we dispose all
second substrips under them, and so on.
Finally, we make the same permutation
of vertical substrips.

Suppose that the direct summands in
\eqref{jer} are numbered so that
\begin{equation}\label{jer7}
p_1>p_2>\dots>p_t.
\end{equation}
Then the block matrix $J^{\#}$ (which
is permutationally similar to a
nilpotent Jordan matrix) is a
\emph{nilpotent Weyr matrix}; see
\cite{mear} or \cite{ser_can}. The
second matrix in \eqref{cev} is block
triangular; in the following lemma we
prove that $S^{\#}$ is  block
triangular for all nilpotent Weyr
matrices. This property is a minor
modification (in the nilpotent case) of
the most important property of Weyr
matrices, which was discovered by
Belitskii \cite{bel} (see also
\cite{bel1,ser_can}): all matrices
commuting with a Weyr matrix are block
triangular.

\begin{lemma}\label{kse1}
{\rm(a)} Let $J^{\#}$ be a nilpotent
Weyr matrix. Then a matrix $X$
satisfies $\bar XJ^{\#}=J^{\#}X$ if and
only if $X=S^{\#}$ for some block
matrix $S$ of the form described in
Lemma \ref{kse}. The matrix $S^{\#}$ is
upper block triangular with respect to
the partition obtained from the
partition of $S$ by the above-described
permutation of substrips.

{\rm(b)} A matrix $S$ of the form
described in Lemma \ref{kse} is
nonsingular if and only if all diagonal
subblocks $C_{ii}$ on its main diagonal
\begin{equation*}\label{dw0}
(C_{11},\bar C_{11},\dots|C_{22},\bar C_{22},\dots
|\dots|C_{tt},\bar C_{tt},\dots)
\end{equation*}
are nonsingular.
\end{lemma}

\begin{proof}
(a) Let $\bar XJ^{\#}=J^{\#}X$. Since
$J^{\#}$ is permutationally similar to
$J$, there is a permutation matrix $P$
such that $J^{\#}=P^{-1}JP$. Since
\[
P\bar XP^{-1}PJ^{\#}P^{-1}=PJ^{\#}P^{-1}PXP^{-1},
\]
we have $\bar SJ=JS$, in which
$S=PXP^{-1}$. Then $X=P^{-1}SP=S^{\#}$
and $S$ has the form described in Lemma
\ref{kse}(b). Only subblocks
$C_{ij}^{(k)}$ ($k=0,1,\dots$) of $S$
can be nonzero. Each subblock
$C_{ij}^{(k)}$ is at the intersection
of horizontal and vertical substrips
indexed by pairs $\alpha,\!i$ and
$\beta\!,\!j$ in which $\beta =\alpha
+k$, hence $\alpha\le\beta$. If
$\alpha=\beta$ then $i\le j$ by
\eqref{jer7}, which proves that
$S^{\#}$  is upper block triangular.

(b) Each matrix $S$ of the form
described in Lemma \ref{kse} is
nonsingular if and only the upper block
triangular matrix $S^{\#}$ is
nonsingular if and only if its diagonal
subblocks $C_{ii}$ and $\bar C_{ii}$
are nonsingular.
\end{proof}

\section{Proof of Theorem \ref{gey}(a)}
\label{s2}

The matrices
\begin{equation}\label{gtr}
J:=\left[ \begin{array}{cccc|c}
0&I&0&0&0\\0&0&I&0&0\\0&0&0&I&0\\0&0&0&0&0\\
\hline 0&0&0&0&0
\end{array}  \right],
     \qquad
M:=\left[ \begin{array}{cccc|c}
0&0&X&0&Y\\0&0&0&\bar X&0\\0&0&0&0&0\\0&0&0&0&0\\
\hline 0&0&0&I&0
\end{array}  \right],
\end{equation}
in which all blocks are $n$-by-$n$ and
the blocks $X$ and $Y$ are arbitrary,
satisfy $\bar MJ=JM$. They define
commuting semilinear operators by Lemma
\ref{kse}(a).

Write
\begin{equation}\label{trw}
M':=\left[ \begin{array}{cccc|c}
0&0&X'&0&Y'\\0&0&0&\overline{ X'}&0\\0&0&0&0&0
\\0&0&0&0&0\\
\hline 0&0&0&I&0
\end{array}  \right].
\end{equation}

The following lemma completes the proof
of  Theorem \ref{gey}(a).

\begin{lemma}\label{lir}
The pairs $(J,M)$ and $(J,M')$ defined
in \eqref{gtr} and \eqref{trw} are
consimilar if and only if $(X,Y)$ and
$(X',Y')$ are consimilar.
\end{lemma}

\begin{proof}
Suppose that there is a nonsingular $S$
such that $\bar S^{-1}(J,M)S=(J,M')$.
Then $JS=\bar SJ$, and by Lemma
\ref{kse}(b)
\[
S=\left[ \begin{array}{cccc|c}
C&C_1&C_2&C_3&D\\0&\overline{C}&\overline{C_1}
&\overline{C_2}&0\\0&0&C&C_1&0\\0&0&0&\overline{C}&0\\
\hline 0&0&0&E&F
\end{array}  \right].
\]
Since $MS=\bar SM'$, we have
\[
\left[ \begin{array}{cccc|c}
0&0&XC&XC_1+YE&YF
\\0&0&0&\overline{X}\overline{C}&0
\\0&0&0&0&0\\0&0&0&0&0\\
\hline 0&0&0&\phantom{\hat{F}}\overline{C}
\phantom{\hat{F}}&0
\end{array}  \right]
     =
\left[ \begin{array}{cccc|c}
0&0&\overline{C}X'
&\overline{C_1}\overline{X'}+\overline{D}
&\overline{C}Y'
\\0&0&0&C\overline{X'}&0
\\0&0&0&0&0\\0&0&0&0&0\\
\hline 0&0&0&\phantom{\hat{F}}\overline{F}
\phantom{\hat{F}}&0
\end{array}  \right],
\]
which implies $XC=\overline{C}X'$,
$YF=\overline{C}Y'$, and
$\overline{C}=\overline{F}$. Hence,
$(X,Y)C=\overline{C}(X',Y')$.

Conversely, if
$(X,Y)C=\overline{C}(X',Y')$ for some
nonsingular $S$, then $(J,M)S=\bar
S(J,M')$ for $S:=\diagg(C,\bar C,C,\bar
C,C)$.
\end{proof}

\section{Proof of Theorem \ref{gey}(b)}
\label{s3}

Let $p$ and $q$ be nonnegative integers
and let $X_1,\dots,X_p,Y_1,\dots,Y_q$
be $n\times n$ matrices. Define the
block matrix
\[
   M_{X,Y}:= \begin{bmatrix}
     0&X_1&&0\\&0&\ddots\\&&\ddots&X_p
     \\0&&&0
   \end{bmatrix}\oplus
  \begin{cases}
     Y_1\oplus Y_2\oplus\dots\oplus Y_q &
     \text{if $p$ is odd}, \\
    0\oplus Y_1\oplus Y_2\oplus\dots\oplus Y_q &
    \text{if $p$ is even},
  \end{cases}
\]
in which all blocks are $n\times n$.
Define the block matrix
\[
J:=\begin{bmatrix}
     0&I_n&&0\\&0&\ddots\\&&\ddots&I_n
\     \\0&&&0\\
   \end{bmatrix}\]
of the same size. Denote by $M_{X',Y'}$
the matrix obtained from $M_{X,Y}$ by
replacing all $X_i$ and $Y_j$ with
$X'_i$ and $Y'_j$.

Theorem \ref{gey}(b) is a consequence
of the following lemma.

\begin{lemma}
The matrix pairs $(J,M_{X,Y})$ and
$(J,M_{X',Y'})$ are consimilar if and
only if there exists a nonsingular $C$
such that
\begin{itemize}
\item[\rm(i)] all $X_{2i}$ are
    similar to $X'_{2i}$ via $C$,

\item[\rm(ii)] all $X_{2i+1}$ are
    similar to $X'_{2i+1}$ via
    $\bar C$,

  \item[\rm(iii)] all $Y_{2i+1}$
      are consimilar to $Y'_{2i+1}$
      via $C$, and
\item[\rm(iv)] all $Y_{2i}$ are
    consimilar to $Y'_{2i}$ via
    $\bar C$.
\end{itemize}
\end{lemma}

\begin{proof} $\Longrightarrow$.
Suppose that there is an $S$ such that
$\bar
S^{-1}(J,M_{X,Y})S=(J,M_{X',Y'})$. By
Lemma \ref{kse}(a), all matrices $S$
satisfying $\bar SJ=JS$ have the form
\begin{equation*}\label{gei}
 S=\begin{bmatrix}
     C & C_1&C_2&C_3&\ddots \\
      & \bar C & \bar C_1&\bar C_2&\ddots \\
      &&C & C_1&\ddots \\
      &&&\bar C & \ddots \\
0&&& & \ddots \\
   \end{bmatrix},
\end{equation*}
and so $\bar S^{-1}M_{X,Y}S=M_{X',Y'}$
implies
\begin{align*}
\bar C^{-1}X_1\bar C&=X'_1,& C^{-1}X_2 C&=X'_2
,& \bar C^{-1}X_3&\bar C=X'_3,\ \dots\\
\bar C^{-1}Y_1C&=Y'_1,& C^{-1}Y_2 \bar C&=Y'_2
,& \bar C^{-1}Y_3& C=Y'_3,\ \dots
\end{align*}
which ensures the validity (i)--(iv).

$\Longleftarrow$. Let (i)--(iv) hold
for some matrix $C$. Then $(J,M_{X,Y})$
and $(J,M_{X',Y'})$ are consimilar via
$S:=C\oplus \bar C\oplus C\oplus \bar
C\oplus \cdots$.
\end{proof}

\section{Proof of Theorem \ref{gey2}}
\label{s4}

In this section, we prove that for each
biquiver $Q$,
\begin{equation}\label{4.8}
\parbox{25em}
{the problem of classifying pairs of
semilinear operators contains the
problem of classifying representations
of $Q$.
}
\end{equation}

To make the proof clear, \emph{we first
establish that \eqref{4.8} holds for
all representations of the biquiver
\eqref{ksy}.}  Its arbitrary
representation $\cal R$ has the form
\eqref{2.6aa}; let the mappings $\cal
A,B,\dots,G$ be given by matrices
$A,B,\dots,G$ in some bases of the
spaces $U,V,W$. Changing the bases, we
can reduce these matrices by
transformations
\begin{equation}\label{aaa}
\begin{split}
\xymatrix@=3pc{
 &{1}&\\
 {2}\ar@(ul,dl)@{-->}_{\bar S_2^{-1}CS_2}
 \ar@{-->}[ur]^{\bar S_1^{-1}AS_2}
  \ar@/^/@{->}[rr]^{S_3^{-1}DS_2}
 \ar@/_/@{-->}[rr]_{\bar S_3^{-1}ES_2} &&{3}
 \ar[ul]_{S_1^{-1}BS_3}
 \ar@(ur,dr)^{S_3^{-1}FS_3}
 }\end{split}
\end{equation}
in which $S_1,S_2,S_3$ are the change
of basis matrices.

Define the matrices
\begin{equation*}\label{jere}
J:=J_{2}(0_{q_1})\oplus J_{7}(0_{q_2})\oplus
J_{4}(0_{q_3}),
\end{equation*}
\begin{equation*}\label{w38}
M:=\left[
\begin{array}{cc|ccccccc|cccc}
0&0& A&0&0&0&0&0&0& 0&0&0&0\\
0&0& 0&0&0&0&0&0&0& B&0&0&0\\
  \hline
0&0& 0&0&C&0&0&0&0& 0&0&0&0\\
0&0& 0&0&0&0&0&0&0& 0&0&0&0\\
0&0& 0&0&0&0&0&0&0& 0&0&0&0\\
0&0& 0&0&0&0&0&0&0& 0&0&0&0\\
0&0& 0&0&0&0&0&0&0& 0&0&0&0\\
0&0& 0&0&0&0&0&0&0& 0&0&0&0\\
0&0& 0&0&0&0&0&0&0& 0&0&0&0\\
\hline
0&0& 0&0&0&0&E&0&0& 0&0&0&0\\
0&0& 0&0&0&0&0&0&D& 0&0&0&0\\
0&0& 0&0&0&0&0&0&0& 0&0&0&0\\
0&0& 0&0&0&0&0&0&0& 0&0&F&0\\
\end{array}\right],
\end{equation*}
and denote by $M'$ the matrix obtained
from $M$ by replacing $A,B,C,D,E,F$
with $A',B',C',D',E',F'$.

The statement \eqref{4.8} is valid for
representations of the biquiver
\eqref{ksy} due to the following lemma.

\begin{lemma}\label{kt5r}
Let $J$, $M,$ and $M'$ be the matrices
defined above. The following statements
are equivalent:
\begin{itemize}
  \item[\rm(i)] The matrix pairs
      $(J,M)$ and $(J,M')$ are
      consimilar.

  \item[\rm(ii)] There exist
      nonsingular matrices
      $S_1,S_2,S_3$ such that
\begin{equation}\label{heo}
\begin{aligned} AS_2&=\bar
S_1A',& \quad BS_3&= S_1B',&\quad
CS_2&=\bar S_2C',\\
DS_2&=S_3D',& ES_2&=\bar
S_3E',&
FS_3&=S_3F'.
\end{aligned}
\end{equation}

\item[\rm(iii)] The matrix tuples
    $(A,B,C,D,E,F)$ and
    $(A',B',C',D',E',F')$ give the
    same representation \eqref{ksy}
    of the biquiver \eqref{2.6aa}
    in different bases; see
    \eqref{aaa}.
\end{itemize}
\end{lemma}

\begin{proof}  (i)\,$\Longrightarrow$\,(ii).
Let $(J,M)$ and $(J,M')$ be consimilar;
that is, there exists a nonsingular
matrix $S$ such that
\begin{equation}\label{ggt}
JS=\bar SJ,\qquad MS=\bar SM'.
\end{equation}
Applying Lemma \ref{kse}(b) to the
first equality in \eqref{ggt}, we
partition $S$ into blocks and subblocks
conformally to the partition of $J$ and
find that the diagonal subblocks of $S$
form a sequence of the form
\[
(S_1,\bar S_1\,|\,S_2,\bar S_2,S_2,\bar S_2,S_2,
\bar S_2,S_2\,|\,S_3,\bar S_3,S_3,\bar S_3).
\]
Lemma \ref{kse1}(b) ensures that the
subblocks $S_1,S_2,S_3$ are
nonsingular. Each of the horizontal and
vertical substrips of $M$ and $M'$ has
at most one nonzero subblock; we obtain
the equalities \eqref{heo} from the
second equality in \eqref{ggt} by
equating the corresponding subblocks on
the positions of subblocks
$A,B,C,D,E,F$.

(i)\,$\Longleftarrow$\,(ii). Suppose
that there are nonsingular matrices
$S_1,S_2,S_3$ that satisfy the
equations \eqref{heo}. Then the
equations \eqref{ggt} are satisfied if
we choose
\[
S:=(S_1\oplus\bar S_1)\oplus (S_2\oplus\bar S_2
\oplus S_2\oplus\bar S_2\oplus S_2\oplus
\bar S_2\oplus S_2)\oplus (S_3\oplus\bar S_3
\oplus S_3\oplus\bar S_3).
\]
It follows that $(J,M)$ and $(J,M')$
are consimilar.

(ii)\,$\Longleftrightarrow$\,(iii).
This equivalence follows from
\eqref{aaa}.
\end{proof}
\bigskip

\begin{proof}[Proof of Theorem \ref{gey2}]
Let us prove \eqref{4.8} for an
arbitrary biquiver $Q$ with vertices
$1,\dots,t$. Let $\cal R$ be a
representation of $Q$. Denote by ${\cal
R}_i$ the vector space that is assigned
to a vertex $i$ and by ${\cal
R}_{\alpha }$ the linear or semilinear
mapping that is assigned to an arrow
$\alpha $. Choose bases in the spaces
${\cal R}_1,\dots,{\cal R}_t$ and
denote by $R_{\alpha }$ the matrix of
${\cal R}_{\alpha }$ in these bases.
Changing the bases, we can reduce all
$R_{\alpha }$ by transformations
\begin{equation}\label{kyw}
R_{\alpha }\mapsto
  \begin{cases}
    S_j^{-1}R_{\alpha }S_i & \text{if $\alpha :i\to j$}, \\
    \bar S_j^{-1}R_{\alpha }S_i & \text{if $\alpha :i\dashrightarrow j$},
  \end{cases}
\end{equation}
in which $S_1,\dots,S_t$ are the change
of basis matrices.

By analogy with the proof of
\eqref{4.8} for the biquiver
\eqref{ksy}, we construct a matrix pair
$(J,M)$ as follows:
\begin{itemize}
  \item The matrix $J$ is any
      matrix of the form
\begin{equation*}\label{te9}
J=J_{p_1}(0_{q_1})\oplus\dots\oplus
J_{p_t}(0_{q_t}),\qquad
p_i\ne p_j\text{ if }i\ne j,\quad
q_i:=\dim {\cal R}_i,
\end{equation*}
in which all $p_i$ are large enough
(it suffices to take $p_i\ge 2n(i)$
in which $n(i)$ is the number of
arrows leaving or entering the
vertex $i$ with loops being counted
twice). The matrix $J$ is divided
into $t$ horizontal and $t$
vertical strips of sizes
$p_1q_1,\dots,p_tq_t$; the $i$th
strip is divided into $p_i$
substrips of size $q_i$.

  \item The matrix $M$ is any
      matrix that satisfies the
      following conditions:
\begin{itemize}
  \item $M$ and $J$ have the
      same size and the same
      partition into horizontal
      and vertical strips and
      substrips,
  \item every substrip of $M$
      has at most one nonzero
      subblock,
  \item the nonzero subblocks
      of $M$ are all the
      nonzero matrices $
      R_{\alpha }$,

  \item if $\alpha$ is an arrow
      from a vertex $i$ to a
      vertex $j$ and $R_{\alpha
      }$ is at the intersection
      of substrip $k$ of
      horizontal strip $i$ with
      substrip $l$ of vertical
      strip $j$, then $k$ is
      even if $\alpha :i\to j$
      and odd if $\alpha
      :i\dashrightarrow j$; $l$
      is odd.
\end{itemize}
\end{itemize}

Reasoning as in the case of the
biquiver \eqref{ksy}, one can prove
that if $(J,M)$ is reduced by
consimilarity transformations that
preserve $J$:
\[
(J,M)\mapsto \bar S^{-1}(J,M)S,
\quad \bar S^{-1}JS=J,\qquad S
\text{ is nonsingular},
\]
then the blocks $R_{\alpha }$ of $M$
are transformed as in \eqref{kyw}.
\end{proof}

\end{document}